\documentclass[12pt]{article}
\usepackage{mathrsfs}
\usepackage{amsfonts}
\usepackage{amssymb,amsmath}
\usepackage{amsthm}
\usepackage[dvips]{graphics}
\usepackage{ae}
\usepackage{bm}
\usepackage{graphicx}
\usepackage{cite}
\usepackage{lineno}
\usepackage{enumitem}
\usepackage{color}

\def\rank{\mbox{\rm rank\,}}

\def\Nul{\mbox{\rm Nul}}

\def\res{\mbox{\,|\,}}

\def\And{\mbox{\rm ~and~}}

\def\For{\mbox{\rm ~for~}}

\def\Col{\mbox{\rm Col\,}}
\def\Par{\mbox{\rm Par}}

\def\sst{\scriptscriptstyle}
\def\de{\delta\hspace{-.3mm}}

\def\Cup{\textstyle\bigcup\limits}
\def\ccup{\textstyle\bigcup}
\def\Cap{\textstyle\bigcap\limits}

\def\Sum{\textstyle\sum\limits}

\def\({\mbox{\rm (}}\def\){\mbox{\rm )}}

\theoremstyle{plain}
\newtheorem{theorem}{Theorem}[section]
\newtheorem{lemma}[theorem]{Lemma}
\newtheorem{proposition}[theorem]{Proposition}
\newtheorem{example}[theorem]{Example}

\newtheorem{corollary}[theorem]{Corollary}

\begin{document}
%\small
\normalsize
%\large
%\huge
\title{Parallel Translates of Represented Matroids}

\author{Beifang Chen\\
\small Department of Mathematics\\
\small Hong Kong University of Science and
Technology\\
\small Clear Water Bay, Hong Kong\\
\small mabfchen@ust.hk
\and
Houshan Fu,
Suijie Wang\thanks{Supported by the National Natural Science Foundation of China (Grant No. 11871204)}\\
\small School of Mathematics\\
\small Hunan University, Changsha, China\\
\small fuhoushan@hnu.edu.cn, wangsuijie@hnu.edu.cn}
\date{}

\maketitle
%\linenumbers

\begin{abstract}
Given an $\Bbb{F}$-represented matroid $(M,\rho)$ with the ground set $[m]$, the representation $\rho$ naturally defines a hyperplane arrangement $\mathcal{A}_\rho$. We will study its parallel translates $\mathcal{A}_{\rho,{\bm g}}$ of $\mathcal{A}_\rho$ for all ${\bm g}\in \Bbb{F}^m$. Its intersection semi-lattices $L(\mathcal{A}_{\rho,{\bm g}})$ and the characteristic polynomials $\chi(\mathcal{A}_{\rho,{\bm g}},t)$ will be classified by the intersection lattice of the derived arrangement $\mathcal{A}_{\de\rho}$, which is a hyperplane arrangement associated with the derived matroid $(\de M,\de\rho)$ and  also known as the discriminantal arrangement in the literature. As a byproduct, we obtain a comparison result and a decomposition formula on the characteristic polynomials $\chi(\mathcal{A}_{\rho,{\bm g}},t)$. \vspace{2ex}

\end{abstract}

\section{Introduction}

Our matroid terminology and notation will follow Oxley \cite{Oxleybook}. For a field ${\Bbb F}$, let $M$ be an $\Bbb{F}$-represented matroid of rank $r$ with the ground set $[m]=\{1,2,\ldots,m\}$, and let $\rho:[m]\rightarrow{\Bbb F}^n$ be a matroid representation of $M$ and write $\rho(i)={\bm\rho}_i$, $i\in[m]$. Denote by the pair $(M,\rho)$ an ${\mathbb F}$-represented matroid. Associated with $(M,\rho)$, the hyperplane arrangement $\mathcal{A}_{\rho}$ consists of hyperplanes $H_i:{\bm\rho}_i{\bm x}=0$ for all $i\in [m]$. We are interested in classifying all parallel translates $\mathcal{A}_{\rho,{\bm g}}=\{H_{{\bm\rho}_1,g_1},\ldots, H_{{\bm\rho}_m,g_m}\}$ of $\mathcal{A}_{\rho}$, where ${\bm g}=(g_1,\ldots,g_m)$ is a vector in ${\Bbb F}^m$ and the hyperplanes $H_{{\bm\rho}_i,g_i}$ are defined by the equations
\begin{equation}\label{Translation-Eq}
H_{{\bm\rho}_i,g_i}:\quad {\bm\rho}_i{\bm x}=g_i, \quad 1\leq i\leq
m.
\end{equation}
A vector is understood as either a row vector or a column vector, depending on its meaning in the context. Our motivation is the work of Athanasiadis \cite{Athanasiadis2010} on deformations of hyperplane arrangements and the work of Oxley-Wang \cite{Oxley-Wang} on derived matroids.

Given a hyperplane arrangement $\mathcal{A}$ in a vector space $V$. The {\bf characteristic polynomial} of $\mathcal{A}$ is
\begin{equation*}
\chi(\mathcal{A},t):=\sum_{X\in L(\mathcal{A})}\mu(V,X)\,t^{\dim X},
\end{equation*}
where $L(\mathcal{A})$ is the intersection semi-lattice whose members are all possible nonempty intersections of hyperplanes of $\mathcal{A}$, including $V$, ordered by the reverse of set inclusion, and $\mu$ is the M\"{o}bius function on $L(\mathcal{A})$; see \cite{Orlik,Stanley1,Stanley2}. Our classification about the parallel tranlates $\mathcal{A}_{\rho,{\bm g}}$ with ${\bm g}\in{\Bbb F}^m$ is to classify the semi-lattices $L(\mathcal{A}_{\rho,{\bm g}})$ and the characteristic polynomials $\chi(\mathcal{A}_{\rho,{\bm g}},t)$. When ${\bm g}={\bm 0}$, denote $\mathcal{A}_{\rho}=\mathcal{A}_{\rho,{\bm 0}}$
whose characteristic polynomial $\chi(\mathcal{A}_{\rho},t)$ is the characteristic polynomial $\chi(M,t)$ of the matroid $M$ (see \cite{Rota,Welsh,Stanley2}). So $\mathcal{A}_{\rho,{\bm g}}$ is a parallel translate of  $\mathcal{A}_{\rho}$.

Rota first introduced the concept of derived matroids to investigate ``dependencies among dependencies" of matroids at the Bowdoin College Summer 1971 NSF Conference on Combinatorics.  In 1980, Longyear \cite{Longyear1980} studied the derived matroids $\de M$ when $M$ is $\Bbb{F}_2$-represented. For any field $\Bbb{F}$, Oxley and Wang \cite{Oxley-Wang} gave a specific definition of the derived matroid of an $\Bbb{F}$-represented matroid $(M,\rho)$. Let  $\mathscr{C}(M)$ denote the set of all circuits of $M$. The derived matroid  $(\de M,\de\rho)$ is an ${\mathbb F}$-represented matroid of rank $r(\de M)=m-r(M)$ with the ground set $\mathscr{C}(M)$ such that
\[
(\de\rho)(I) = {\bf c}_{\sst I}\quad \For \quad I\in \mathscr{C}(M),
\]
where the {\bf circuit vector} ${\bf c}_{\sst I}= (c_1,c_2,\dots,c_m)$ in ${\mathbb F}^m$ is unique (up to a non-zero scalar multiple) and defined by  $\sum_{i = 1}^m c_i {\bm \rho}_i = 0$, where  $c_i \neq 0$ if and only if $i \in I$.  Similar as $\mathcal{A}_{\rho}$, the hyperplane arrangement $\mathcal{A}_{\de\rho}$ in $\Bbb{F}^m$ is defined by
\begin{equation}\label{derivedarrangement}
\mathcal{A}_{\de\rho}=\{H_{\sst I}\mid I\in \mathscr{C}(M)\},\quad\quad {\rm where} \;\; H_{\sst I}: \; {\bm c}_{\sst I}{\bm x}=0
\end{equation}
and called {\bf the  derived arrangement} of $(M,\rho)$.

In 1989, Manin and Schechtman \cite{Manin-Schechtman1989} introduced the discriminantal arrangement to characterize the general position parallel translates of an affine hyperplane arrangement in general position, known as Manin-Schechtman arrangement. In 1994, Falk \cite{Falk1994} shown that neither combinatorial nor topological structure of the discriminantal arrangement is independent of the original arrangement, that is, in matroid language, the derived matroid $(\de M,\de\rho)$ depends on the representation $\rho$ of the original matroid $M$, confirmed in \cite[page 4-6]{Oxley-Wang}. In 1997, Bayer and Brandt \cite{Bayer-Brandt1997} studied the discriminantal arrangments without general position assumptions on the original arrangement. It is easily seen from \cite[Theorem 2.4]{Bayer-Brandt1997} that the discriminantal arrangement is exactly the derived arrangement defined above. Please see \cite{Athanasiadis1999, Libgober-Settepanella2018, Sawada-Settepanella-Yamagata2017} 
for more study on discriminantal arrangements.

For each $X\in L(\mathcal{A}_{\de\rho})$, consider the restriction of
$\mathcal{A}_{\de\rho}$ to $X$, denoted $\mathcal{A}_{\de\rho}/X$, which is a hyperplane arrangement in $X$ consisting of the hyperplanes $H_{\sst I}\cap X$ of $X$, where
$H_{\sst I}\in\mathcal{A}_{\de\rho}$ and $X\not\subseteq H_{\sst I}$. Let $M(\mathcal{A}_{\de\rho}/X)=X\setminus \ccup_{{\scriptscriptstyle H_{\sst I}:X\nsubseteq H_{\sst I}\in \mathcal{A}_{\de\rho}}}H $ be the complement of $\mathcal{A}_{\de\rho}/X$. We have the
disjoint decomposition \cite[page 410]{Stanley2}
\begin{equation}\label{DecompositionRho}
{\Bbb F}^m=\bigsqcup_{X\in L(\mathcal{A}_{\de\rho})}
M(\mathcal{A}_{\de\rho}/X).
\end{equation}
Indeed, for each ${\bm x}\in {\Bbb F}^m$, assuming $X=\bigcap_{{\bm x}\in H_{\sst I}\in \mathcal{A}_{\de\rho}} H_{\sst I}$, we have ${\bm x}\in M(\mathcal{A}_{\de\rho}/X)$. Below are our main results.
\begin{theorem}[Geometric Characterization]\label{GCT}
If vectors ${\bm g},{\bm h}\in  M(\mathcal{A}_{\de\rho}/X)$ for some  $X\in L(\mathcal{A}_{\de\rho})$, then
\begin{equation*}\label{geometric characterization}
L(\mathcal{A}_{\rho,{\bm g}})\cong L(\mathcal{A}_{\rho,{\bm h}})\quad \And\quad
\chi(\mathcal{A}_{\rho,{\bm g}},t)=\chi(\mathcal{A}_{\rho,{\bm h}},t).
\end{equation*}
\end{theorem}
When ${\bm g}\in M(\mathcal{A}_{\de\rho})$, the arrangement $\mathcal{A}_{\rho,{\bm g}}$ is in general position, exactly the same as the arrangement studied by Bayer and Brandt \cite{Bayer-Brandt1997}.
Theorem \ref{GCT} says that $\chi(\mathcal{A}_{\rho,{\bm g}},t)$ are the same polynomial for all ${\bm
g}\in M(\mathcal{A}_{\de\rho}/X)$, denoted by $\chi(X,t)$. For
convenience let us write
\begin{eqnarray}
\chi(\mathcal{A}_{\rho,{\bm g}},t)=\chi(X,t)= \sum_{k=0}^r (-1)^k a_k(X) t^{n-k},\quad {\bm
g}\in M(\mathcal{A}_{\de\rho}/X). \label{Defn-Chi-X}
\end{eqnarray}

\begin{theorem}[Uniform Comparison Theorem]\label{comparison}
Under the notations of \eqref{Defn-Chi-X}, if
$X\subseteq Y$ for $X,Y\in L(\mathcal{A}_{\de\rho})$, then
\begin{equation*}
a_k(X)\leq a_k(Y), \quad 0\leq k\leq r.
\end{equation*}
\end{theorem}

\begin{theorem}[Decomposition Formula]\label{decomposition}
Under the notations of \eqref{Defn-Chi-X}, if $\Bbb F$ is a finite field of $q$ elements,
then
\begin{equation*}
\sum_{X\in L(\mathcal{A}_{\de\rho})} \chi(X,q)
\chi(\mathcal{A}_{\de\rho}/X,q)=q^n(q-1)^m.
\end{equation*}
If $\Bbb F$ is an infinite field, then
\begin{equation*}\label{Poly-Id}
\sum_{X\in L(\mathcal{A}_{\de\rho})} \chi(X,t)
\chi(\mathcal{A}_{\de\rho}/X,t)=t^n(t-1)^m.
\end{equation*}
\end{theorem}
To end this section, we give a small example to illustrate the above theorems.
\begin{example}
Consider the uniform matroid $U_{2,4}$ represented over $\Bbb{R}$ by \[{\bm\rho_1}=(1,0),\;{\bm\rho_2}=(0,1),\;{\bm\rho_3}=(1,1),\; {\bm\rho_4}=(1,-1).\]
The associated arrangement $\mathcal{A}_{\rho}$ consists of hyperplanes $H_i:{\bm \rho}_i{\bm x}=0$ in $\Bbb{R}^2$.
For convenience we assume that the ground set of $U_{2,4}$ is $E=\{{\bm\rho_i}\mid i=1,2,3,4\}$. It is obvious that for each $i$, $I_i=E\smallsetminus \{{\bm\rho_i}\}$ is a circuit. The corresponding circuit vectors are
\[
{\bm c}_{\sst I_1}=(0,2,-1,1),\;{\bm c}_{\sst I_2}=(2,0,-1,-1),\;{\bm c}_{\sst I_3}=(1,-1,0,-1),\;{\bm  c}_{\sst I_4}=(1,1,-1,0),
\]
and the derived arrangement is $\mathcal{A}_{\delta\rho}=\{H_{\sst I_i}: {\bm c}_{\sst I_i}{\bm x}=0\mid i=1,2,3,4\}$ in $\Bbb{R}^4$. The intersection lattice $L(\mathcal{A_{\delta\rho}})$ consists of flats
$V=\Bbb{R}^4,H_{\sst I_1}, H_{\sst I_2},H_{\sst I_3},H_{\sst I_4}$, $W=\Cap_{i=1}^4 H_{\sst I_i}$ with $W\subseteq H_{\sst I_i}\subseteq V$. By routine calculations, we can obtain the characteristic polynomials
\[
\chi(\mathcal{A}_{\delta\rho}/V,t)=t^4-4t^3+3t^2,\; \chi(\mathcal{A}_{\delta\rho}/H_{\sst I_i},t)=t^3-t^2,\; \chi(\mathcal{A}_{\delta\rho}/W,t)=t^2,
\]
and
\[
\chi(V,t)=t^2-4t+6,\; \chi(H_{\sst I_i},t)=t^2-4t+5,\;\chi(W,t)=t^2-4t+3.
\]
As Theorem \ref{comparison} and \ref{decomposition}, we have the coefficient comparison $a_k(W)\le a_k(H_{\sst I_i}) \le a_k(V)$ for $k=0,1,2,i=1,2,3,4$, and the decomposition formula
\[
\sum_{X\in L(\mathcal{A}_{\de\rho})} \chi(X,t)
\chi(\mathcal{A}_{\de\rho}/X,t)=t^2(t-1)^4.
\]
\end{example}

\section{Proofs of Main Results}

Let $\mathcal{A}=\{H_{1},\ldots,H_{m}\}$ be a hyperplane arrangement
in an $n$-dimensional vector space $V$. A subset $J\subseteq[m]$ is
said to be {\bf affine independent} (with respect to $\mathcal{A}$)
if $\cap_{j\in J}H_j\neq \emptyset$ and
\[
r(\cap_{j\in J}H_j):=n-\dim(\mbox{$\cap$}_{j\in J}H_j)=|J|.
\]
Likewise, a subset $J$ of $[m]$ is said to be {\bf affine dependent}
if $\cap_{j\in J}H_j\neq \emptyset$ and $r(\cap_{j\in J}H_j)<|J|$.
Subsets $J\subseteq[m]$ with $\cap_{j\in J}H_j=\emptyset$ are
irrelevant to affine independence and affine dependence. An {\bf
affine circuit} is a minimal affine dependent subset $I$ of $[m]$,
that is, $I$ is affine dependent and any proper subset of $I$ is
affine independent. It is easy to see that a subset $I$ of $[m]$ is
an affine circuit if and only if $\cap_{i\in I}H_i\neq\emptyset$ and
for each $i_0\in I$,
\begin{equation*}
r(\cap_{i\in I}H_i)=r(\cap_{i\in I\smallsetminus i_0}H_i)=|I|-1.
\end{equation*}
Given a total order $\prec$ on $[m]$. An {\bf affine broken circuit}
is a subset of $[m]$ obtained from an affine circuit by removing its
maximal element under the total order $\prec$. A subset $J$ of $[m]$
is called an {\bf affine NBC} (no-broken-circuit) if $\cap_{j\in
J}H_j\neq \emptyset$ and $J$ contains no affine broken circuits. Of
course all affine NBC subsets are affine independent. The affine NBC
sets are just the $\chi$-independent sets in \cite[p.72]{Orlik}.

\begin{theorem}[Affine Broken Circuit Theorem\cite{Orlik}]\label{ABC}
Let $r=\rank(\mathcal{A})$ and write $\chi(\mathcal{A},t)$ as
\[
\chi(\mathcal{A},t)=a_0t^n-a_1t^{n-1}+\cdots + (-1)^ra_rt^{n-r}.
\]
Then the coefficient $a_k$ equals the number of affine NBC
$k$-subsets with respect to $\mathcal{A}$, $0\le k\le r$.
\end{theorem}

Given an ${\mathbb F}$-represented matroid $(M,\rho)$ with the ground set $[m]$ and the representation $\rho(i)={\bm\rho}_i$ for $i\in[m]$, let $A$ denote the matrix whose rows are the vectors
${\bm\rho}_1,\ldots,{\bm\rho}_m$. For each subset $J\subseteq[m]$,
let $[A\,|\,J]$ denote the matrix whose rows are those rows of $A$
having the row indices in $J$. For each vector ${\bm g}\in{\Bbb
F}^m$, we define
\begin{equation*}
\mathcal{J}_{\bm g}=\{J\subseteq[m]: [A\res J]\,{\bm x} =[{\bm g}
\res J]\;\mbox{is consistent}\}.
\end{equation*}
We assume that $\mathcal{J}_{\bm g}$ contains the empty set
$\emptyset$. Then $\mathcal{J}_{\bm g}$ is a lower order ideal of
the Boolean lattice $(2^{[m]},\subseteq)$. Recall the hyperplane arrangements $\mathcal{A}_{\rho,{\bm g}}$ defined by
\eqref{Translation-Eq}, where ${\bm g}\in{\Bbb F}^m$. By definition
of $\mathcal{J}_{\bm g}$ and by affine dependence and affine
independence,
\begin{equation}\label{J-Union}
\mathcal{J}_{\bm g}=\{\mbox{affine dependent sets}\} \cup
\{\mbox{affine independent sets}\}
\end{equation}
with respect to $\mathcal{A}_{\rho,{\bm g}}$. Let $\mathscr{C}_{\bm g}=\mathcal{J}_{\bm
g}\cap\mathscr{C}(M)$, where ${\bm g}\in{\Bbb F}^m$. We claim that
\begin{equation*}\label{C-Union}
\mathscr{C}_{\bm g}=\{\mbox{affine circuits with respect to
$\mathcal{A}_{\rho,{\bm g}}$}\}.
\end{equation*}
Note that (i) a subset $J$ of $[m]$ is affine independent with
respect to $\mathcal{A}_{\rho,{\bm g}}$ if and only if $J$ is independent
in $M$; (ii) any affine dependent subset with respect to
$\mathcal{A}_{\rho,{\bm g}}$ is dependent in $M$ (its converse is
not necessarily true). Let $J\in \mathscr{C}_{\bm
g}=\mathcal{J}_{\bm g}\cap\mathscr{C}(M)$. Clearly, $J$
must be affine dependent with respect to $\mathcal{A}_{\rho,{\bm g}}$. For
each $j_0\in J$, since $J\smallsetminus j_0$ is independent in
$M$, then $J\smallsetminus j_0$ is affine independent with
respect to $\mathcal{A}_{\rho,{\bm g}}$ by (i). So $J$ is an affine circuit
with respect to $\mathcal{A}_{\rho,{\bm g}}$. Conversely, let $J$ be an
affine circuit with respect to $\mathcal{A}_{\rho,{\bm g}}$. By definition
$J$ is affine dependent and $J\smallsetminus j_0$ is affine
independent with respect to $\mathcal{A}_{\rho,{\bm g}}$ for all $j_0\in
J$. Then $J\in\mathcal{J}_{\bm g}$ by \eqref{J-Union}, $J$ is
dependent in $M$ by (ii), and $J\smallsetminus j_0$ is
independent by (i) in $M$ for each $j_0\in J$. The latter
two mean that $J\in\mathscr{C}(M)$. Hence  $J\in
\mathscr{C}_{\bm g}$.

\begin{lemma}\label{Circuit-Hyperplane-Characterization}
For each circuit $I\in\mathscr{C}(M)$, the hyperplane
$H_{\sst I}$ defined by \eqref{derivedarrangement} can be
written as
\[
H_{\sst I}=\{{\bm g}\in{\Bbb F}^m: [A\res I]\,{\bm x} =[{\bm g} \res
I]\;\mbox{\rm is consistent}\}.
\]
In other words, a vector ${\bm g}$ belongs to $H_{\sst I}$ if and only if
its restriction $[{\bm g} \res I]$ belongs to the column space $\Col
[A\res I]$ of the submatrix $[A\res I]$.
\end{lemma}
\begin{proof}
Recall the circuit vector ${\bm c}_{\sst I}$ associated with the circuit
$I$. Since $c_i=0$ for all $i\in[m]\smallsetminus I$, then for the
restrictions of both ${\bm c}_{\sst I}$ and a vector ${\bm g}$ to $I$, we
have
\begin{equation}\label{Circuit-Restriction}
{\bm c}_{\sst I}A=[{\bm c}_{\sst I}\res I]\,[A\res I]={\bm 0}, \quad {\bm c}_{\sst I}{\bm g} =[{\bm
c}_{\sst I}\res I]\,[{\bm g} \res I].
\end{equation}
If the system $[A\res I]\,{\bm x} =[{\bm g} \res I]$ is consistent,
then
\[
{\bm c}_{\sst I}{\bm g} =[{\bm c}_{\sst I}\res I]\,[{\bm g} \res I]=[{\bm c}_{\sst I}\res
I]\,[A\res I]\,{\bm x}  = 0.
\]
This means that the vector ${\bm g}$ belongs to $H_{\sst I}$.

Conversely, given ${\bm g}\in H_{\sst I}$, we have $[{\bm c}_{\sst I}\res
I]\,[{\bm g} \res I]={\bm c}_{\sst I}{\bm g} =0$. Then $[{\bm g}\res I]$
belongs to the solution space $\Nul\, [{\bm c}_{\sst I}\res I]$ of the
single equation $[{\bm c}_{\sst I}\res I]\,{\bm x} =0$, and $\Nul\, [{\bm
c}_{\sst I}\res I]$ has dimension $|I|-1$. Since $I$ is a circuit, the
matrix $[A\res I]$ has rank $|I|-1$; then $\Col [A\res I]$ also has
dimension $|I|-1$. Note that the first part of
\eqref{Circuit-Restriction} implies $\Col [A\res I]\subseteq \Nul\,
[{\bm c}_{\sst I}\res I]$. Hence $\Col [A\res I]=\Nul\, [{\bm c}_{\sst I}\res I]$.
This means that $[{\bm g} \res I]\in\Col[A\res I]$.
\end{proof}
Lemma~\ref{Circuit-Hyperplane-Characterization} implies
\begin{equation}\label{CH}
\mathscr{C}_{\bm g}=\{I\in\mathscr{C}(M): {\bm g}\in
H_{\sst I}\}.
\end{equation}
Recall the disjoint decomposition \eqref{DecompositionRho}. For each vector ${\bm g}\in{\Bbb F}^m$, there exists a unique
$X\in L(\mathcal{A}_{\de\rho})$ such that ${\bm g}\in
M(\mathcal{A}_{\de\rho}/X)$; we denote this unique space $X$ associated
with $\bm g$ by $X_{\bm g}$. For convenience of later discussion we
state the following fact:
\begin{equation}\label{I-C_g-X_g}
I\in\mathscr{C}_{\bm g}\; \Leftrightarrow\; {\bm g}\in H_{\sst I}
\;\Leftrightarrow\; X_{\bm g}\subseteq H_{\sst I}.
\end{equation}
The first equivalence is a restatement of \eqref{CH}. For the second
equivalence, note that we have ${\bm g}\in M(\mathcal{A}_{\de\rho}/X_{\bm
g})$ by definition of $X_{\bm g}$, that is,
\[
{\bm g} \in \bigcap_{H_{\sst J}\in\mathcal{A}_{\de\rho},\,X_{\bm g}\subseteq
H_{\sst J}}H_{\sst J}-\bigcup_{H_{\sst J}\in\mathcal{A}_{\de\rho},\,X_{\bm g}\not\subseteq
H_{\sst J}}H_{\sst J}.
\]
Now purely logically, ${\bm g} \in
\bigcap_{H_{\sst J}\in\mathcal{A}_{\de\rho},\,X_{\bm g}\subseteq H_{\sst J}}H_{\sst J}$ means
that if $X_{\bm g}\subseteq H_{\sst I}$ then ${\bm g}\in H_{\sst I}$. And ${\bm
g}\not\in \bigcup_{H_{\sst J}\in\mathcal{A}_{\de\rho},\,X_{\bm g}\not\subseteq
H_{\sst J}}H_{\sst J}$ means that if $X_{\bm g}\not\subseteq H_{\sst I}$ then ${\bm
g}\not\in H_{\sst I}$.

\begin{lemma}\label{Consistency-Minimal-Dependence-Consistency}
A linear system $A{\bm x}={\bm g}$ is consistent if and only if the
subsystems
\[
[A\res I]\,{\bm x}=[{\bm g}\res I]
\]
are consistent for all those $I\subseteq [m]$ that the rows of
$[A\res I]$ form a minimal linearly dependent set.
\end{lemma}
\begin{proof}
The necessity is trivial. For sufficiency, recall that $A{\bm
x}={\bm g}$ is consistent if and only if its coefficient matrix $A$
and its augmented matrix $[A,{\bm g}]$ have the same rank. We choose a maximal linearly independent set of rows of $A$ with row
index set $J$. Then $\rank [A\res J]=|J|=\rank A$, and the rows of
$[A\res J\cup i]$ are linearly dependent for each
$i\in[m]\smallsetminus J$. Thus there exists a minimal row index set
$J_i\subseteq J$ such that the rows of $[A\res J_i\cup i]$ are
linearly dependent, that is, the set of rows of $[A\res J_i\cup i]$
is a minimal linearly dependent set. So the system $[A\res J_i\cup
i]\,{\bm x}=[{\bm g}\res J_i\cup i]$ is consistent by given
conditions. This means that
\[
\rank[A\res J_i\cup i]=\rank[A\res J_i\cup i,{\bm g}\res J_i\cup i].
\]
Since $\rank [A\res J_i\cup i]=\rank [A\res J_i]$, we obtain $\rank
[A\res J_i]=\rank [A\res J_i\cup i,{\bm g}\res J_i\cup i]$. Now we
have
\[
\rank [A\res J_i,{\bm g}\res J_i] \leq \rank [A\res J_i\cup i, {\bm
g}\res J_i\cup i]=\rank [A\res J_i] \leq \rank [A\res J_i, {\bm
g}\res J_i].
\]
It follows that
\[
\rank [A\res J_i,{\bm g}\res J_i] = \rank [A\res J_i\cup i, {\bm
g}\res J_i\cup i].
\]
This means that the $i$th row of $[A,{\bm g}]$ is a linear
combination of the rows of $[A\res {J_i},{\bm g}\res {J_i}]$, and of
course a linear combination of the rows of $[A\res J,{\bm g}\res
J]$, where $i\in[m]\smallsetminus J$. Since the rows of $[A\res J]$
are linearly independent, the system $[A\res J]{\bm x}=[{\bm g}\res
J]$ is automatically consistent.  We see that the system $A{\bm
x}={\bm g}$ is consistent.
\end{proof}

\begin{corollary}\label{CJ}
Let ${\bm g},{\bm h}\in{\Bbb F}^m$. Then $\mathcal{J}_{\bm
g}\subseteq \mathcal{J}_{\bm h}$ if and only if  $\mathscr{C}_{\bm
g}\subseteq \mathscr{C}_{\bm h}$. In particular, $\mathcal{J}_{\bm
g}= \mathcal{J}_{\bm h}$ if and only if  $\mathscr{C}_{\bm g}=
\mathscr{C}_{\bm h}$.
\end{corollary}
\begin{proof}
The necessity is trivial. For sufficiency, given a subset
$J\subseteq [m]$. By
Lemma~\ref{Consistency-Minimal-Dependence-Consistency}, the system
$[A\res J]\,{\bm x} =[{\bm g} \res J]$ is consistent if and only if
$[A\res I]\,{\bm x} =[{\bm g} \res I]$ is consistent for all those
$I\subseteq J$ such that the rows of $[A\res I]$ form a minimal
linearly dependent set, that is, if and only if
$I\in\mathscr{C}_{\bm g}$ whenever $I\in\mathscr{C}(M)$
and $I\subseteq J$. By definitions of $\mathcal{J}_{\bm g}$ and
$\mathscr{C}_{\bm g}$,
%$J\in\mathcal{J}_{\bm g}$ is equivalent to $I\in\mathscr{C}_{\bm g}$
%for all circuits $I$ of $M$ with $I\subseteq J$,
\[
\mathcal{J}_{\bm g}=\{J\subseteq[m]: I\in\mathscr{C}_{\bm
g}\;\mbox{for all}\;
I\in\mathscr{C}(M)\;\mbox{with}\;I\subseteq J\}.
\]
Since $\mathscr{C}_{\bm g}\subseteq \mathscr{C}_{\bm h}$, it is
clear that $\mathcal{J}_{\bm g}\subseteq \mathcal{J}_{\bm h}$.
\end{proof}
\begin{proof}[Proof of Theorem \ref{GCT}]
By definition of $X_{\bm h}$, ${\bm h}\in M(\mathcal{A}_{\de\rho}/X_{\bm g})$ is  equivalent to $X_{\bm h}=X_{\bm g}$. Since
$X=\bigcap\{H_{\sst I}\in\mathcal{A}_{\de\rho}: X\subseteq H_{\sst I}\}$ for each $X\in
L(\mathcal{A}_{\de\rho})$, then $X_{\bm g}= X_{\bm h}$ is equivalent to
\begin{equation}\label{X_g-X_h-Containment}
\{H_{\sst I}\in\mathcal{A}_{\de\rho}: X_{\bm g}\subseteq
H_{\sst I}\}=\{H_{\sst I}\in\mathcal{A}_{\de\rho}: X_{\bm h}\subseteq H_{\sst I}\}.
\end{equation}
Applying \eqref{I-C_g-X_g}, then \eqref{X_g-X_h-Containment} is
equivalent to
\begin{equation*}
\{H_{\sst I}\in\mathcal{A}_{\de\rho}: {\bm g}\in H_{\sst I}\}=\{H_{\sst I}\in\mathcal{A}_{\de\rho}:
{\bm h}\in H_{\sst I}\},
\end{equation*}
which is further equivalent to
\begin{equation}\label{I-C}
\{I\in\mathscr{C}(M): {\bm g}\in
H_{\sst I}\}=\{I\in\mathscr{C}(M): {\bm h}\in H_{\sst I}\}.
\end{equation}
Applying \eqref{CH}, we see that \eqref{I-C} is equivalent to
$\mathscr{C}_{\bm g}=\mathscr{C}_{\bm h}$, which is further
equivalent to $\mathcal{J}_{\bm g}=\mathcal{J}_{\bm h}$ by
Corollary~\ref{CJ}. Summarizing the above discussion, we have
\[
X_{\bm g}=X_{\bm h}\; \Leftrightarrow\; \mathscr{C}_{\bm
g}=\mathscr{C}_{\bm h} \; \Leftrightarrow\; \mathcal{J}_{\bm
g}=\mathcal{J}_{\bm h}.
\]
Note that
\[
L(\mathcal{A}_{\rho,{\bm g}})=\left\{\Cap_{i\in J}H_{{\bm \rho}_i,g_i}\mid J\in \mathcal{J}_{\bm
g} \right\},\quad\quad L(\mathcal{A}_{\rho,{\bm h}})=\left\{\Cap_{i\in J}H_{{\bm \rho}_i,h_i}\mid J\in \mathcal{J}_{\bm
h} \right\},
\]
and $\dim(\cap_{i\in J}H_{{\bm \rho}_i,g_i})=\dim(\cap_{i\in J}H_{{\bm \rho}_i,h_i})$ if $J\in \mathcal{J}_{\bm g},\mathcal{J}_{\bm h}$.
This means that ${\bm h}\in M(\mathcal{A}_{\de\rho}/X_{\bm g})$ implies $L(\mathcal{A}_{\rho,{\bm g}})\cong L(\mathcal{A}_{\rho,{\bm h}})$ and $\chi(\mathcal{A}_{\rho,{\bm g}},t)=\chi(\mathcal{A}_{\rho,{\bm h}},t)$.
\end{proof}
\begin{proof}[Proof of Theorem \ref{comparison}]
Write $X=X_{\bm g}$ and $Y=X_{\bm h}$ for some ${\bm g},{\bm
h}\in{\Bbb F}^m$, i.e., ${\bm g}\in M(\mathcal{A}_{\de\rho}/X)$ and ${\bm h}\in M(\mathcal{A}_{\de\rho}/Y)$. Applying \eqref{CH} and \eqref{I-C_g-X_g}, we have
\[
\mathscr{C}_{\bm g} = \{I\in\mathscr{C}(M): {\bm g}\in
H_{\sst I}\} = \{I\in\mathscr{C}(M): X_{\bm g}\subseteq H_{\sst I}\}.
\]
Since $X_{\bm g}\subseteq X_{\bm h}$, it follows that
$\mathscr{C}_{\bm h}\subseteq \mathscr{C}_{\bm g}$. Let
$B\mathscr{C}_{\bm g}$ denote the set of (affine) broken circuits
obtained from the (affine) circuits in $\mathscr{C}_{\bm g}$. Then
$B\mathscr{C}_{\bm h}\subseteq B\mathscr{C}_{\bm g}$. Let
$NB\mathscr{C}_{\bm g}$ denote the set of all affine NBC sets with
respect to $\mathcal{A}_{\rho,{\bm g}}$. Let $J\in NB\mathscr{C}_{\bm g}$.
Then $J$ is affine independent with respect to $\mathcal{A}_{\rho,{\bm g}}$
and $J$ does not contain elements of $B\mathscr{C}_{\bm g}$. Since
affine independence is equivalent to independence in $M$,
we see that $J$ is affine independent with respect to
$\mathcal{A}_{\rho,{\bm h}}$ and of course $J$ does not contain elements of
$B\mathscr{C}_{\bm h}$, that is, $J\in NB\mathscr{C}_{\bm h}$. Thus
$NB\mathscr{C}_{\bm g}\subseteq NB\mathscr{C}_{\bm h}$. Recall the
Affine Broken Circuit Theorem~\ref{ABC} that $a_k(X)$ equals
the number of $k$-subsets of $NB\mathscr{C}_{\bm g}$. We see that
$a_k(X)\leq a_k(Y)$,$0\leq k\leq r$.
\end{proof}
It remains to prove Theorem \ref{decomposition} in this section. Roughly speaking, the characteristic polynomial of a hyperplane arrangement measures the size of its complement. Let $V$ be a vector space over an infinite field ${\Bbb F}$, and let
$\mathcal{B}(V)$ be the Boolean algebra generated by affine
subspaces through intersection, union, and complement finitely many
times. A {\bf valuation} on $\mathcal{B}(V)$ is a map $\rho$ from
$\mathcal{B}(V)$ to an abelian group such that
$\rho(\emptyset)=0$ and $\rho(A\cup
B)=\rho(A)+\rho(B)-\rho(A\cap B)$ for
$A,B\in\mathcal{B}(V)$. It is known \cite{BFC,Ehrebborg} that there
exists a unique (translation-invariant) valuation
$\nu:\mathcal{B}(V)\rightarrow{\Bbb Z}[t]$ such that $\nu(W)=t^{\dim
W}$ for all affine subspaces $W$ of $V$. Given a hyperplane
arrangement $\mathcal{A}$ on the vector space $V$. Note that the complement
$M(\mathcal{A})$ is a member of $\mathcal{B}(V)$ and its valuation
under $\nu$ is the characteristic polynomial of $\mathcal{A}$, that
is,
\[
\chi(\mathcal{A},t)=\nu(M(\mathcal{A})).
\]
The valuation $\nu$ defines an integral
\[
\int f{\rm d}\nu = \sum_{i=1}^\ell c_i \nu({X_i})
\]
for each simple function $f=\sum_{i=1}^\ell c_i 1_{X_i}$, where
$X_i\in\mathcal{B}(V)$ and $1_{X_i}$ is the indicator function of
$X_i$. The following proposition is a Fubini-type theorem.

\begin{proposition}[Ehrenborg and Readdy \cite{Ehrebborg}]\label{proposition-fubini-type}
Let $\nu_1,\nu_2$ and $\nu$ be the unique valuations on vector
spaces $V_1,V_2$ and $V_1\times V_2$, respectively. Let $f(x,y)$ be a
simple function on $V_1\times V_2$. Then $f_x(y)$ is a simple
function on $V_2$ for each $x\in V_1$, and $\int f_x(y){\rm d}\nu_2$
is a simple function on $V_1$. Moreover,
\[
\int f(x,y){\rm d}\nu=\int\int f_x(y){\rm d}\nu_2{\rm d}\nu_1.
\]
\end{proposition}
\begin{proof}[Proof of Theorem \ref{decomposition}] Consider the hyperplane arrangement $\tilde{\mathcal{A}}$ in ${\Bbb
F}^n\times {\Bbb F}^m$ with the hyperplanes
\[
{\bm\rho}_i{\bm x} =y_i,\quad 1\leq i\leq m,
\]
where $({\bm x},{\bm y})$ is the vector variable with ${\bm
x}=(x_1,\ldots,x_n)$ and ${\bm y}=(y_1,\ldots,y_m)$. The system of
the $m$ defining equations are linearly independent. The
intersection semi-lattice of $\tilde{\mathcal{A}}$ is isomorphic to
the Boolean lattice $(2^{[m]},\subseteq)$. So
$\chi(\tilde{\mathcal{A}},t)=t^n(t-1)^m$. On the other hand, since ${\Bbb F}^m=\bigsqcup_{X\in
L(\mathcal{A}_{\de\rho})} M(\mathcal{A}_{\de\rho}/X)$ by \eqref{DecompositionRho}, we have
\[
M(\tilde{\mathcal{A}}) = \bigsqcup_{X\in L(\mathcal{A}_{\de\rho})}
M(\tilde{\mathcal{A}})\cap ({\Bbb F}^n\times M(\mathcal{A}_{\de\rho}/X)).
\]
Let $\Omega(X)= M(\tilde{\mathcal{A}})\cap ({\Bbb F}^n\times
M(\mathcal{A}_{\de\rho}/X))$, which is a member of the Boolean algebra
$\mathcal{B}({\Bbb F}^{n+m})$. Then the indicator function
$1_{\Omega(X)}$ is a simple function with respect to
$\mathcal{B}({\Bbb F}^{n+m})$. Note that $({\bm x}, {\bm g})\in \Omega(X)$ means ${\bm g}\in M(\mathcal{A}_{\de\rho}/X)$ and ${\bm\rho}_i{\bm x} \ne g_i,\; 1\leq i\leq m$, i.e.,
\[
\Omega(X)=\bigsqcup_{{\bm g}\in M(\mathcal{A}_{\de\rho}/X)}
M(\mathcal{A}_{\rho,{\bm g}})\times\{{\bm g}\}.
\]
Whenever $\Bbb F$ is an infinite field, note that
$\nu_1(M(\mathcal{A}_{\rho,{\bm g}}))=\chi(\mathcal{A}_{\rho,{\bm
g}},t)=\chi(X,t)$ for all ${\bm g}\in M(\mathcal{A}_{\de\rho}/X)$.
Applying Proposition~\ref{proposition-fubini-type},
\[
\int 1_{\Omega(X)}{\rm d}\nu = \int\int 1_{\Omega(X)}{\rm
d}\nu_1{\rm d}\nu_2 = \chi(X,t)\, \chi(\mathcal{A}_{\de\rho}/X,t).
\]
Hence
\[
\chi(M(\tilde{\mathcal{A}},t))=\sum_{X\in L(\mathcal{A}_{\de\rho})}
\chi(X,t) \chi(\mathcal{A}_{\de\rho}/X,t).
\]
Whenever ${\Bbb F}$ is a finite field of $q$ elements, the
valuations $\nu_1,\nu_2$ and $\nu$ are understood as counting
measures, and the variable $t$ is replaced by the number $q$.
\end{proof}

\section{Applications}

Consider $M=U_{1,n}$, the uniform matroid with the ground set $[n]$ and of rank $1$. For any $\Bbb{F}$- representation  $\rho$ of $M$, we have $\de M\cong M(K_n)$, the cycle matroid of the complete graph $K_n$. Without loss of generality, assume  the representation $\rho: [n]\to \Bbb{F}$ with $\rho(i)={\bm \rho}_i=1$.  Its derived arrangement $\mathcal{A}_{\de\rho}$ is
\[
\mathcal{A}_{\de\rho}=\{H_{i,j}\mid 1\le i<j\le n\}\quad{\rm where}\; H_{i,j}: x_i=x_j.
\]
We say $P=\{S_1,\ldots, S_k\}$ a partition of $[n]$ if $S_i\cap S_j=\emptyset$ and $[n]=\Cup_{i=1}^kS_i$, and denote by $l(P)=k$ the number of parts in $P$. For two partitions $P=\{S_1,\ldots, S_k\}$ and $P'=\{S'_1,\ldots, S'_{k'}\}$ of [n], we say $P'$ is a refinement of $P$ if for each $i\in [k']$, $S'_i$ is a subset of $S_j$ for some $j\in [k]$. Let $(\Par(n), \le)$ be the poset consisting of all partitions of $[n]$ and ordered by refinements. It is obvious that
\[
L(\mathcal{A}_{\de \rho})\cong L(\de M)= L(M(K_n))\cong \Par(n).
\]
More precisely, for each $X\in L(\mathcal{A}_{\de \rho})$,  there is a unique partition $P_{\sst X}=\{S_1,\ldots, S_k\}$ such that
\[
X=\{(x_1\ldots, x_n)\in \Bbb{F}^m\mid x_i=x_{i'} {\rm ~whenever~} i,i'\in S_j {\rm ~for ~some~} j\in [k] \}.
\]
which follows
\[
\chi(\mathcal{A}_{\de \rho}/X, t)=t(t-1)\cdots (t-k+1).
\]
For any vector ${\bm g}=(g_1,\ldots,g_n)\in {\Bbb F}^m$, the parallel translate $\mathcal{A}_{\rho,{\bm g}}$ is
\[
\mathcal{A}_{\rho,{\bm g}}=\{H_{{\bm\rho}_i,g_i}\mid 1\le i\le n\},\quad{\rm where}\;
H_{{\bm\rho}_i,g_i}: {\bm\rho}_i x=x=g_i,
\]
and
\[
\chi(\mathcal{A}_{\rho,{\bm g}},t)=t-\delta_{\bm g},
\]
where  $\delta_{\bm g}=\#\{g_i\mid i\in [n]\}$ is the number of distinct entries of ${\bm g}$. If $X\in L(\mathcal{A}_{\de \rho})$ and $P_{\sst X}=\{S_1,\ldots, S_k\}$, then ${\bm g}\in M(\mathcal{A}_{\de \rho}/X)$ implies  $\delta_{\bm g}=k$ and $\chi(X,t)=t-k$. From Theorem \ref{decomposition}, we immediately have the following identity
\[
\sum_{P\in {{\rm Par}(n)}}t(t-1)\cdots (t-l(P))=t(t-1)^n.
\]
Note that the Stirling number $S(n, k)$ of the second kind   counts the number of partitions of $[n]$ into $k$ parts. Removing $t$ on both sides and replacing $t$ by $x+1$, we obtain the well known identity  on $S(n, k)$ below
\[
\sum_{k\ge 0}S(n,k)x(x-1)\cdots(x-k+1)=x^n.
\]


\begin{thebibliography}{99}
%\bibitem{Athanasiadis}
%C.A. Athanasiadis, Characteristic polynomials of subspace
%arrangements and finite fields, Adv. in Math. 122 (1996) 193-233.
\bibitem{Athanasiadis1999} C. A. Athanasiadis, The largest intersection lattice of a discriminantal arrangement, Beitr. Algebra Geom, 40 (1999), no. 2, 283-289.


\bibitem{Athanasiadis2010} C. A. Athanasiadis, A combinatorial
reciprocity theorem for hyperplane arrangements, Canad. Math. Bull.
53 (2010), 3-10.

\bibitem{Bayer-Brandt1997} M. Bayer, K. A. Brandt, Discriminantal Arrangements, Fiber Polytopes and Formality, J. Algebraic Combin., 6 (1997), 229-246.

\bibitem{BFC}
B. Chen, On characteristic polynomials of subspace arrangements, J.
Combin. Theory Ser. A 90 (2000), 347-352.

\bibitem{Ehrebborg}
R. Ehrenborg, M. A. Readdy, On valuations, the characteristic
polynomial, and complex subspace arrangements, Adv. in Math. 134
(1998), 32-42.

\bibitem{Falk1994}
M. Falk, A note on discriminantal arrangements, Pro. Amer. Math. Soc., 122 (1994), 1221-1227.

\bibitem{Libgober-Settepanella2018} A. Libgober, S. Settepanella, Strata of discriminantal arrangements, J. Singul. 18 (2018), 440-454.

\bibitem{Longyear1980}
J. Q. Longyear, The circuit basis in binary matroids, J. Number Theory 12 (1980), 71-76.

\bibitem{Manin-Schechtman1989}
Y. I. Manin, V. V. Schechtman, Arrangements of Hyperplanes, Higher Braid Groups and Higher Bruhat Orders, In: Algebraic Number Theory-in honor of K. Iwasawa, Advanced Studies in Pure Math., 17 (1989), 289-308.

\bibitem{Orlik}
P. Orlik, H. Terao, Arrangements of Hyperplanes, Springer, 1992.

\bibitem{Oxleybook}
J. Oxley, Matroid Theory, Second edition,  Oxford University Press, New York, 2011.

\bibitem{Oxley-Wang}
J. Oxley, S. Wang, Dependencies among Dependencies in Matroids, Electron. J. Combin. 26(3) (2019), \#P3.46.

\bibitem{Rota}
G.-C. Rota, On the foundations of combinatorial theory I: Theory of
M\"{o}bius functions, Zeitschrift f\"{u}r Wahrscheinlichkeitstheorie
2 (1964), 340-368.

\bibitem{Sawada-Settepanella-Yamagata2017} S. Sawada, S. Settepanella, S. Yamagata, Discriminantal arrangement, $3\times 3$ minors of Pl\"{u}cker matrix and hypersurfaces in Grassmannian Gr(3,n),
    C. R. Math. Acad. Sci. Paris, 355 (2017), no. 11, 1111-1120.

\bibitem{Stanley1}
R. P. Stanley, Enumerative Combinatorics I, Cambridge Univ. Press, Cambridge, 1997.

\bibitem{Stanley2}
R. P. Stanley, An introduction to hyperplane arrangements, in: E.
Miller, V. Reiner, B. Sturmfels (Eds.), Geometric Combinatorics, in:
IAS/Park City Math. Ser., vol. 13, Amer. Math. Soc., Providence, RI,
2007, pp. 389-496.

\bibitem{Welsh}
D. J. A. Welsh, Matroid Theory. Academic Press, London, New York, San
Francisco, 1976.
\end{thebibliography}
\end{document}